\documentclass[12pt, a4paper]{amsart}
\usepackage[T1]{fontenc}
\usepackage[utf8]{inputenc}
\usepackage[english]{babel}
\usepackage{amsmath, amssymb, amsthm, amsfonts,latexsym,graphicx}

 \newcommand{\fq}{\ensuremath{\mathbb{F}_q}}
 \newcommand{\tm}{\ensuremath{\mathbf{T}\left( m, \mathbb{F}_q \right) }} 
  
 \newcommand{\fqm}{\ensuremath{\mathbb{F}_q^\times}}
 \newcommand{\szp}{\ensuremath{\mathbf{SP}(m,\mathbb{F}_q)}} 
 \newcommand{\Sl}{\ensuremath{\mathbf{SL}\left( 2, \mathbb{Z}_3 \right) }} 
 \newcommand{\Z}{\ensuremath{\mathbf{Z}}}
 \newcommand{\Zb}{\ensuremath{\mathbb{Z}}}
 \newcommand{\UT}{\ensuremath{{\mathbf{U} \left( 3, \mathbb{Z}_3 \right)}}}
 \newcommand{\UTsz}{\ensuremath{{\mathbf{U} \left( 3, \mathbb{Z}_3 \right)}\rtimes{\ensuremath{\mathbb{Z}_3^\times}}}}
 \newcommand{\zh}{\ensuremath{\mathbb{Z}_3}}
 \newcommand{\R}{\ensuremath{\mathcal{R} }} 
 \newcommand{\M}{\ensuremath{\mathcal{M} }} 
 \newcommand{\I}{\ensuremath{\mathcal{I} }} 
 \newcommand{\G}{\ensuremath{\mathbf{G}}} 
 \newcommand{\D}{\ensuremath{\mathbf{D}_n}} 
 \newcommand{\A}{\ensuremath{\mathbf{A}}} 
 \newcommand{\B}{\ensuremath{\mathbf{B}}} 
 \newcommand{\Sim}{\ensuremath{\mathbf{S}}} 
 \newcommand{\N}{\ensuremath{\mathbf{N}}} 
 \newcommand{\sg}{\ensuremath{\mathbf{S}}} 
 \newcommand{\sus}{\ensuremath{S}} 
 \newcommand{\di}{\ensuremath{\mathbf{D}}}  
 
\newcommand{\za}{\ensuremath{\mathbb{Z}_{p^{\alpha}} }} 
 
\newcommand{\al}{\ensuremath{\alpha}}

\theoremstyle{plain}
\newtheorem{theo}{Theorem}[section]
\newtheorem{lem}[theo]{Lemma}

\theoremstyle{definition}

\begin{document}

\title[Equation solvability over semipattern groups]{The complexity of the equation solvability problem over semipattern groups}
\author[A.~F\"oldv\'ari]{Attila F\"oldv\'ari}
\address{
University of Debrecen,
Institute of Mathematics,
Pf.~400, Debrecen, 4002, Hungary
}
\email{foldvari.attila@science.unideb.hu}
\thanks{This research was partially supported by the Hungarian National Foundation for Scientific Research grant no.~K109185.}
\date{15 March, 2016}

\begin{abstract}
The complexity of the equation solvability problem is known for nilpotent
groups, for not solvable groups and for some semidirect products of 
Abelian groups.
We provide a new polynomial time algorithm for deciding the equation solvability problem
over certain semidirect products, where the first factor is not necessarily Abelian. 
Our main idea is to represent such groups as matrix groups, 
and reduce the original problem to 
equation solvability over the underlying field. 
Further, 
we apply this new method to give a much more efficient algorithm for equation solvability over nilpotent rings than previously existed. 
\end{abstract}



\maketitle

\section{Introduction}

One of the oldest problems of algebra is the equation solvability problem over a given algebraic structure.  
Nowadays, 
many such classical problems arise in a new perspective, 
namely to consider their computational complexity. 
In this paper we investigate the complexity of the equation solvability problem over finite groups and rings. 

The \emph{equation solvability problem} over a finite group $ \G $  asks whether or not two group expressions (i.e.\ products of variables and elements of $ \G $) can attain the same value for some substitution over $ \G $.
In other words, 
for the equation solvability problem, 
one needs to find if there exists at least one substitution satisfying the equation. 
Another interesting problem is whether or not \emph{all} substitutions satisfy the equation. 
The \emph{equivalence problem} over a finite group $ \G $  asks whether or not two group expressions $ f $ and $ g $ are equivalent over $ \G $ (denoted by $\G\models f \approx g $), 
that is whether or not  $ f $ and $ g $ determine the same function over $ \G $.
 
First Burris and Lawrence \cite{Burris} investigated the complexity of the equivalence problem over finite groups. 
They proved that if a group $ \G $ is nilpotent or $ \G \simeq \D $, the dihedral group for odd $ n $, then the equivalence problem for $ \G $ has polynomial time complexity. 
They conjectured that the equivalence problem for $ \G $ is in polynomial time if $ \G $ is solvable, and coNP-complete otherwise. 
Horv\'{a}th and Szab\'o \cite{HG3} confirmed the conjecture for $ \G \simeq \A \rtimes \B $, where $ \A $ and $ \B $ are Abelian groups such that the exponent of $ \A $ is squarefree and $ (|\A| , |\B|) = 1 $.
Later Horv\'{a}th \cite{HG15JA} generalized this result to semidirect products $ \A \rtimes \B $, where $ \A $ and $\B/C_\B(\A)$ are Abelian groups (here $C_\B(\A)$ denotes the centralizer of $ \A$ in $\B$). 
Horv\'ath, Lawrence, M\'erai and Szab\'o \cite{HLM1} proved the coNP-complete part of the conjecture. 
But the complexity of the equivalence problem over many solvable, not nilpotent groups is not determined, yet. 
Three of the smallest groups, for which this complexity is not known, are $\Sim_4$, $\Sl$ and a non-commutative group of order $54$. 
See \cite{HG15JA} for a more comprehensive list. 

Even less is known about the equation solvability problem.
Goldmann and Russel \cite{goldmann, russel} proved that if $ \G $ is nilpotent then the equation solvability problem over $ \G $ is solvable in polynomial time, while if $ \G $ is not solvable, then the equation solvability problem is NP-complete. 
Little is known for solvable, not nilpotent groups.
Horv\'ath proved in \cite[Corollary 2]{HG15JA} that the equation solvability problem over $\G$ is solvable in polynomial time for certain groups $\G \simeq \A \rtimes \B$, where $\A \simeq \Z_{p^k}$ or $\Z_{2  p^k}$ or $\Z_p^k $ and $\B$ is commutative. 
Note that all results for both the equivalence and the equation solvability problem over solvable, not nilpotent groups are about groups $\A \rtimes \B$, where $\A$ is Abelian. 

One of the groups of small order, for which the equation solvability problem is unknown, is the group $\UTsz$.
Here, 
$\UT$ denotes the noncommutative group of $3 \times 3$ upper unitriangular matrices over $\zh$. 
Horv\'ath explicitly asks in \cite[Problem~4]{HG15JA} the complexity of the equivalence and equation solvability problems over this group. 
The group $\UTsz$ is isomorphic to a special subgroup of the $ 3 \times 3 $ upper triangular matrices over $ \zh$. 
Motivated by the definition of pattern groups from \cite{pattern2, pattern1}, 
we call a group $\A \rtimes \B$ a \emph{semipattern} group, 
if $\A$ is a subgroup of the group of upper unitriangular matrices, 
and $\B$ is a subgroup of the diagonal matrices.
We give the precise definition of semipattern groups in Section~\ref{spcs}.
The main result of the paper is the following. 
 
\begin{theo}
	\label{fotetel}
	The equation solvability problem over semipattern groups is solvable in polynomial time.
\end{theo}

The group $\UTsz$ defined in \cite[Problem~4]{HG15JA} is in fact a semipattern group,
thus Theorem~\ref{fotetel} answers Horv\'ath's question completely. 
Further, 
from Theorem~\ref{fotetel} the equivalence problem over semipattern groups is solvable in polynomial time, as well.
Indeed, 
it is known that for a group $\G$ if the equation solvability problem is solvable in polynomial time, 
then the equivalence problem is solvable in polynomial time, as well.

In the proof of Theorem~\ref{fotetel} we reduce the solvability of the input equation over a matrix group over a finite field to the solvability of a system of equations over the same field.
Then we apply some results over finite rings. 
Therefore, we summarize the known results over rings. 

The \emph{equation solvability problem}
over a finite ring $ \R $  asks whether or not two polynomials can attain the same value for some substitution over $ \R $.
The \emph{equivalence problem}
over a finite ring $ \R $  asks whether or not two polynomials are equivalent over $ \R $ i.e.\ if they determine the same function over $ \R $.

The complexity of these questions was completely characterized in the past two decades. 
Hunt and Stearnes \cite{HHRS90} investigated the equivalence problem for finite commutative rings. 
Later Burris and Lawrence \cite{BSLJ93} generalized their result to non-commutative rings. 
They proved that the equivalence problem for $\R$ is solvable in polynomial time if $\R$ is nilpotent, and is coNP-complete otherwise.  

The proof of Burris and Lawrence reduces the satisfiability (SAT) problem to the equivalence problem by using long products of sums of  variables. 
Nevertheless, if we expand this polynomial into a sum of monomials then the length of the new polynomial may become exponential in the length of the original polynomial. 
Such a change in the length suggests that the complexity of the equivalence problem might be different if the input polynomials are restricted to be written as sums of monomials.
This motivated Lawrence and Willard \cite{LJWR97} to introduce the \emph{sigma equivalence} and \emph{sigma equation solvability problems}, 
where the input polynomials are given as sums of monomials. 
Lawrence and Willard conjectured that if the factor by the Jacobson radical is commutative then the sigma equivalence problem is solvable in polynomial time, and is coNP-complete otherwise. 
Szab\'o and V\'ertesi proved the coNP-complete part of the conjecture in \cite{SZCSVV11}.
Horv\'ath confirmed the conjecture for commutative rings in \cite{HG12GMJ}.
The polynomial part of this conjecture is completely proved in the manuscript \cite{HGLJWR15}. 

Most of the results for the [sigma] equation solvability problem follow from the corresponding result for the [sigma]
equivalence problem.
In particular, 
from the argument of Szab\'o and V\'ertesi follows that if the factor by the Jacobson radical is not commutative then the sigma equation solvability problem is NP-complete.
Horv\'ath, Lawrence and Willard \cite{HGLJWR15} proved that if this factor is commutative then the sigma equation solvability problem is solvable in polynomial time.  
Thus, the sigma equation solvability problem is completely characterized. 

For the general equation solvability, arguments of Burris and Lawrence from \cite{BSLJ93} yield that if the ring is not nilpotent then the problem is NP-complete. 
Horv\'ath in \cite{HG11} proved that the equation solvability problem is solvable in polynomial time otherwise. 

\begin{theo}[{\cite[Theorem~1.2]{HG11}}]
	\label{nr}
	If $ \R $ is a finite, nilpotent ring then the equation solvability problem over $ \R $ is solvable in polynomial time. 
\end{theo}
 

Horv\'ath uses Ramsey's theorem in the proof of Theorem~\ref{nr}. 
He defines a number $r$ that depends only on the ring $\R$. 
Then he proves that the image of every polynomial can be obtained by substituting $ 0 $ into all but $r$-many variables.
Thus one can decide whether or not $f=0$ is solvable over $\R$ in $O\left( ||f||^r\right) $ time. 
However, this number $ r $ is huge in the size of the ring. 
In fact, $r$ is greater than $|\R|^{|\R|^{\dots^{|\R|}}}$, where the height of the tower in the exponent  is the nilpotency class of $\R$.
Horv\'ath specifically asks in \cite[Problem~3]{HG11}
whether or not this number $r$ can be decreased. 

In the second half of the paper we give a new proof of Theorem~\ref{nr}.
Our algorithm is much more efficient than Horv\'ath's. 
Wilson \cite{RW73CM} characterizes nilpotent rings with the help of special kind of nilpotent matrix rings. 
We can decide the equation solvability problem over these special matrix rings similarly as we do over semipattern groups in Theorem~\ref{fotetel}. 
In particular, we show that over a nilpotent ring $\R$ we can decide in $O\left( ||f||^{ |\R|^{2\log \R} \log^5|\R| }\right) $ time whether or not $f=0$ is solvable, 
thus providing a partial answer to Problem~3 in \cite{HG11}, as well. 
We mention that in a completely independent way, 
K\'arolyi and Szab\'o also found a way to decrease the exponent $r$ in \cite{karszab}.

In Section~\ref{bev} we summarize the notations, definitions and theorems, that we use in the paper. 
In particular, in Section~\ref{spcs} we review the definition of pattern groups, then we define semipattern groups.
In Section~\ref{2.3} we discuss the generalization of the equation solvability problem over rings for systems of equations. 
We are going to apply these results in order to prove Theorems~\ref{fotetel}~and~\ref{nr}.
In Section~\ref{nil} we lay the groundwork for the proof of Theorem~\ref{nr}. 
In Section~\ref{2.} we prove Theorem~\ref{fotetel}. 
We use ideas of this proof in Section~\ref{3.}, where we prove Theorem~\ref{nr}. 


\section{Preliminaries \label{bev}}


\subsection{Semipattern groups\label{spcs}} 
Let $\fq$ denote the finite field of $q$ elements.
Let us consider the group $\tm$ of $m\times m$ upper triangular matrices, that is those matrices whose  elements under the diagonal are zero and the elements in the diagonal are non-zero:

\[
\tm = \left\{ \begin{pmatrix}
s_1 & a_{1,2}& a_{1,3} & \dots & a_{1,m} \\
0 & s_2 & a_{2,3} & \dots & a_{2,m}  \\ 
0 & 0 & s_3 & \dots & a_{3,m} \\ 
\vdots & \vdots & \vdots & \ddots & \vdots  \\
0 & 0 & 0 & \dots & s_m \end{pmatrix} 
: s_i\in \fqm,  a_{i,j}  \in \fq , 1 \leq i < j \leq m \right\}.
\] 
Here the group operation is the matrix multiplication. 

Let the $m\times m$ identity matrix denoted by  $I$.
Let $I_{i,j}$ denote the $m\times m$ matrix whose elements are all zero except for the $j^{th}$ element in the $i^{th}$ row, which is 1. 
Let $P \subseteq\{(i,j) : 1 \leq i<j \leq m\}$.
Let 
\[
\N_P =
\{ I+ \sum_{(i,j) \in P} a_{i,j} I_{i,j}: a_{i,j}\in \fq \}.
\]
Thus, $\N_P$ contains all those upper triangular matrices, where every element in the diagonal is $1$, and every element whose position is not occurring in P has to be $0$. 
If $\N_P$ is a subgroup of $\tm$, then we call $\N_P$ a \emph{pattern} group.
For more details on pattern groups, see e.g.~\cite{pattern2,pattern1}.
Let $\sg_1,\sg_2, \dots , \sg_m$ be subgroups of $\fqm$ and let $\di$ be the set of $m\times m$ matrices over $\fq$ whose $i^{th}$ element in the diagonal is from $\sg_i$ $(1\leq i \leq m)$:
\[
\mathbf{D}=
\left\{ 
\begin{pmatrix}
s_1 &0&0 & \dots & 0 \\
0 & s_2 & 0 & \dots & 0  \\ 
0 & 0 & s_3 & \dots &0 \\ 
\vdots & \vdots & \vdots & \ddots & \vdots  \\
0 & 0 & 0 & \dots & s_m \end{pmatrix} 
: s_1\in \sg_1,  s_2\in \sg_2, \dots,  s_m\in \sg_m  \right\}.
\] 

If $\N_P$ is a pattern group then $\N_P\di$ is a subgroup of $\tm$.
Then we call $\N_P\di$ a \emph{semipattern group}  and we denote such a group by $\szp$.
Further, we note that $\N_P \lhd \N_P\di$ and $ \N_P\di \cong \N_P \rtimes \di$. 
The group $\UTsz$ defined in \cite[Problem~4]{HG15JA} is in fact a semipattern group:
\[
\UTsz = \left\{ 
\begin{pmatrix}
1 & a& b  \\
0 & s & c   \\ 
0 & 0 & 1 & 
\end{pmatrix} 
: s\in \zh^\times,  a, b, c  \in \zh \right\}.
\]

\subsection{Further notations\label{2.3}} 
Let ${\R}$ be a commutative, unital ring,  $\sus_1, \dots$ $\dots,\sus_m$ be  subsets of ${\R}$.
For nonnegative integers $n,  l_1, \dots l_m$, let  
$X=\left\{ \, x_1, \dots, x_n \, \right\}$, 
$Y_1=\left\{ \, y_{1,1}, \dots , y_{1,l_1}\, \right\}$, $\dots $, 
$Y_m=\left\{ \, y_{m,1}, \dots, y_{m,l_m} \, \right\}$ 
be pairwise disjoint sets  of variables. 
We say that \emph{${f=g}$ is solvable over ${\R}$ for substitutions from $\R,\sus_1, \dots, \sus_m$} (and write $ f|_{\R,\sus_1, \dots, \sus_m} $ $= g|_{\R, \sus_1, \dots, \sus_m}$ is solvable over $\R$) if there exist 
$a_1,\dots a_n \in{\R}$, 
$s_{1,1},\dots s_{1,l_1}\in \sus_1$, 
$\dots$, 
$s_{m,1}, \dots, s_{m,l_m}\in \sus_m$ 
such that the two polynomials attain the same value on this substitution:
\begin{align*}
f(a_1,\dots a_n ,  s_{1,1},\dots s_{1,l_1} , \dots, s_{m,1}, \dots, &s_{m,l_m} )=\\
g(a_1,\dots & a_n,  s_{1,1},\dots s_{1,l_1} , \dots, s_{m,1}, \dots, s_{m,l_m})\\
\end{align*}
For proving Theorem~\ref{fotetel}, 
we will directly apply the following result of Horváth \cite{HG15JA}. 
\begin{theo}[{\cite[p.~221, case~(d)]{HG15JA}}]
	\label{egyrszmo}
	Let $\fq$ be a finite field.
	Let $\sg_1, \dots{}, \sg_m$ be subgroups of $\fqm$.
	Let 
$f_1, \dots{}, f_k \in \fq[x_1, \dots x_n, y_{1,1}, \dots, y_{m,l_m} ]$ 
be a polynomial, written as a sum of monomials.
	Then it can be decided whether the system of equations
	\begin{align*}
	f_{1} &| _{\fq, \sg_1, \dots{}, \sg_m}= 0 \\
	& \vdots{}\\
	f_k &|_{\fq, \sg_1, \dots{}, \sg_m}= 0\\ \end{align*} is solvable over $\fq$ in $O\left( max_{1\leq i\leq k}||f_i||^{k\cdot q}\right) $ time.
\end{theo}
To prove Theorem~\ref{nr} we are going to use the following from \cite{HGLJWR15}.
\begin{theo}[{\cite{HGLJWR15}}]
	\label{er}
	Let $f_1, \dots{}, f_k \in \Zb_{p^{\alpha}}[y_1, \dots{} , y_n ]$ be  polynomials, written as sums of monomials.  
	Then it can be decided whether the system of equations
	\begin{align*}
	f_1&=0 \\
	&\vdots\\
	f_k&=0
	\end{align*}
	is solvable over $\Zb_{p^{\alpha}}$ in $O\left( max_{1\leq i \leq k} ||f_i||^{\al^2  k \cdot p^{ 2\alpha^2 }}\right) $ time. 
\end{theo}

\subsection{Equation solvability problem over nilpotent rings\label{nil}}
The complexity of the equation solvability problem over finite nilpotent rings is known.
Horv\'{a}th \cite{HG11} proved that the equation solvability problem is solvable in polynomial time.
We give a new algorithm in Section~\ref{3.} that is much more efficient than Horv\'ath's.
In this part we show that we can characterize the complexity of the  equation solvability problem over nilpotent rings using the sigma equation solvability problem over special kind of nilpotent matrix rings.   
Hence we can apply the same ideas that we used in the proof of Theorem~\ref{fotetel}. 

 
Horv\'{a}th proved Theorem~\ref{nr} using Ramsey's theory.
He defined a number $r$ that depends only on the ring $\R$. 
Then he proved that the image of every polynomial can be obtained by substituting $0$ into all but $r$-many variables.
Thus one can decide whether or not $f=0$ is solvable over $\R$ in $O(||f||^r)$ time.
But the number $ r $ is huge in the size of the ring. 
Let $c$ be the characteristic of $\R$ and $t$ be the nilpotency class of $\R$. 
Let furthermore $m=(t-1)!\cdot c$. 
Then $r$ is greater than $m^{m^{\dots^{m}}}$, where the height of the tower in the exponent  is $t$.  
We give a new proof of Theorem~\ref{nr}  in Section~\ref{3.}.
Our algorithm is much more efficient than Horv\'ath's. 
We prove that $O\left( ||f||^{|\R|^{2\log {|\R|}} \log^5{|\R|} }\right) $ time is enough. 
   
First, we show that we can handle the equation solvability problem over nilpotent rings using the sigma equation solvability problem over special kind of nilpotent matrix rings. 
If $\R$ is a nilpotent ring then the complexity of the \emph{general} equation solvability problem over $\R$ is the same as the complexity of the \emph{sigma} equation solvability over $\R$.  
Indeed, we can rewrite every polynomial $f$ over $\R$ as a sum of monomials  in $O\left( ||f||^t\right) $ time, where $t$ is the nilpotency class of $\R$.
Now, 
$t = O \left( \log \left| \R \right| \right)$.
Hence, at the cost of an extra $\log \left| \R \right| $ factor in the exponent, 
we may assume that every input polynomial is given as a sum of monomials. 
Furthermore it is enough to consider the equation solvability problem over nilpotent rings with prime power characteristic, because every ring is a direct sum of rings of prime power characteristic, and the equation solvability problem can be handled componentwise.
Wilson \cite{RW73CM} characterizes finite nilpotent rings with prime power characteristic. 

\begin{theo}[Wilson]
	\label{wil}
	Let $\R$ be a finite nilpotent ring with characteristic $p^{\alpha}$, and let $\R$ have an independent generating set consisting of $m$ generators over $\mathbb{Z}_{p^{\alpha}}$.
	Then $\R$ is a homomorphic image of a ring $\M(m,\mathbb{Z}_{p^{\alpha}})$ of matrices over $\mathbb{Z}_{p^{\alpha}}$ where every entry on or below the main diagonal is a multiple of $p$. 
	 
\end{theo}

In fact for nilpotent rings it is enough to consider the equation solvability problem over such a matrix ring $\M$, 
since if the equation solvability problem is solvable in polynomial time for $\M$, 
then it is solvable in polynomial time for a factor $\M / \I$, as well. 
Indeed, let  $\widetilde{f}$ be a polynomial over $\M/\I$ and $f$ be any polynomial over $\M$ whose factor by $\I$ is $\widetilde{f}$. 
Then $\widetilde{f}=0$ is solvable over $\M/\I$  if and only if $ f=a $  is solvable over $\M$ for some $a\in \I$. 
This gives an extra $\left| \I \right|$ factor to the running time. 
However, 
$\left| \I \right| \leq \left| \M \right| \leq \left(p^\alpha\right)^{m^2}$, 
and we have $p^\alpha = O \left( \left| \R \right| \right)$, 
$m = O \left( \log \left| \R \right| \right)$. 
Thus, $\left| \I \right|  = O \left( \left| \R \right|^{\log^2 \left| \R \right|} \right)$, 
which only depends on $\R$, 
and thus can be forgotten about.

Thus it is enough to consider the sigma equation solvability problem over such a matrix ring of matrices over $\Zb_{p^{\alpha}}$ where every entry on or below the main diagonal of each matrix is a multiple of $p$. 
We can handle such matrix rings similarly as we do with the semmipattern groups. 
Hence we can use ideas of the proof described in Section~\ref{2.}.  
We give a new, efficient algorithm that decides the equation solvability problem over nilpotent matrix rings in Section~\ref{3.}.

\section{Equation solvability problem over semipattern groups \label{2.}\label{2.2}}

In this section we consider the equation solvability problem over semmipattern groups.
First we characterize the multiplication of matrices from $\tm$  in Lemma~\ref{mxszor}. 
We use this formula in our algorithm. 

\begin{lem}
	\label{mxszor}
	Let n be a natural number. For every $1 \leq  k \leq n$ let
	\[
	A_k =  \begin{pmatrix}
	s_{1,k} & a_{1,2,k}& a_{1,3,k} & \dots & a_{1,m,k} \\ 0 & s_{2,k} & a_{2,3,k} & \dots & a_{2,m,k}  \\ 0 & 0 & s_{3,k }& \dots & a_{3,m,k} \\  \vdots & \vdots & \vdots & \ddots & \vdots  \\ 0 & 0 & 0 & \dots & s_{m,k} \end{pmatrix} 
	\in \tm.
	\] 
	Let
	\[
	A_1A_2\dots{}A_n =  
	\begin{pmatrix}
	\sigma_{1} & \alpha_{1,2}& \alpha_{1,3} & \dots & \alpha_{1,m} \\ 
	0 & \sigma_{2} & \alpha_{2,3} & \dots & \alpha_{2,m}  \\ 
	0 & 0 & \sigma_{3}& \dots & \alpha_{3,m} \\ 
	\vdots & \vdots & \vdots & \ddots & \vdots  \\ 
	0 & 0 & 0 & \dots & \sigma_{m}\end{pmatrix} .
	\]
	Then 
	\begin{itemize}
		\item   for every $i=1,2, \dots, n$ we have
		\[\sigma_i= \prod_{k=1}^{n}s_{i,k};
		\]
		\item for every $1 \leq i < j \leq m$ we have
		
		\begin{align}
		\label{keplet}
		\notag
		\alpha_{i,j}=
		&\sum_{b=0}^{j-i-1}    
		\sum_{i<l_1<\dots<l_b<j}   
		\sum_{k_{b+1}=b+1}^{n}
		\sum_{k_{b}=b}^{k_{b+1}-1}    
		\dots
		\sum_{k_{2}=2}^{k_{3}-1}
		\sum_{k_{1}=1}^{k_{2}-1}
		\left(\prod_ {c_0=1}^{k_1-1} s_{i,c_0}\right) a_{i,l_1,k_1}
		\\
		&\left(\prod_ {c_1=k_1+1}^{k_{2}-1} s_{l_1,c_1}\right)
		a_{l_{1},l_{2},k_{2}} \left(\prod_ {c_{2}=k_{2}+1}^{k_{3}-1}     s_{l_{2},c_{2}}\right)
		a_{l_{2},l_{3},k_{3}}  \left(\prod_ {c_{3}=k_{3}+1}^{k_{4}-1}     s_{l_{3},c_{3}}\right)
		\\
		\notag
		&\dots{}  a_{l_{b-1},l_{b},k_{b}}  \left( \prod_ {c_{b}=k_{b}+1}^{k_{b+1}-1}     s_{l_{b},c_{b}}\right)
		a_{l_b,j,k_{b+1}}       \left( \prod_ {c_{b+1}=k_{b+1}+1}^{n} s_{j,c_{b+1}}\right).
		\end{align}
	\end{itemize}  
	The length of \eqref{keplet} is $O\left( n^m\right) $, and in particular is polynomial in $n$. 
\end{lem}

\begin{proof}
The lemma can be proved by induction on $n$.
However, instead of giving the technical induction proof, we  explain how one can arrive at this formula.

Let us consider  the $j^{th}$ element of the $i^{th}$ row $\alpha_{i,j}$ of the matrix $A_1A_2\dots{}A_n  $ $( 1 \leq i < j \leq m)$ .
We can express this element $\alpha_{i,j}$ with a sum of some appropriate products.
In every such product we multiply one element from each matrix $A_k$ $(1\leq k \leq n)$. 
(In the notation the index $k$ appears as the last index.) 
Furthermore the  index of the column of a term must equal with the  index of the row of the following term in every such product.
(Therefore a term of the form $s_{l_c,.}$ or $a_{l_c,l_{c+1},.}$ follows the term of the form $s_{l_c,.}$ or $a_{l_{c-1},l_{c},.}$.)
The row index of the first term is $i$, the  column index of the last term is $j$. 
The matrices $A_1, A_2, \dots{}, A_n $ are upper triangular matrices, hence the row  index of every term of every product is less than or equal to the column  index. 
(Thus for every term of the form $a_{l_c,l_{c+1},.}$, that are above the diagonal, we have $l_c< l_{c+1}$.)
Thus the $j^{th}$ element of the $i^{th}$ row $\alpha_{i,j}$ of the matrix $A_1A_2\dots{}A_n$  is a sum of products of $n$ terms such that 
\begin{itemize}
	\item the row index of the first term is $i$, the column index of the last term is $j$;
	\item the column  index of every term  equals to the  row index of the next term;
	\item the  row index of a term  is  at most the  row index of the next term.
\end{itemize}

Notice, that every $n$-term product is uniquely determined by those terms where the row index differs from the column index.  
Let two such consecutive terms of  the product be $a_{l_{c-1},l_{c},k_{c}}$ and $a_{l_{c},l_{c+1},k_{c+1}}$.
(Here $i \leq l_{c-1}< l_{c}< l_{c+1} \leq j$ and $1\leq k_c<k_{c+1}\leq n$.)
The product of the terms between these two terms have row and column index $l_c$ and their last index is more than $k_c$ and less than $k_{c+1}$.
Thus between the terms $a_{l_{c-1},l_{c},k_{c}}$ and $a_{l_{c},l_{c+1},k_{c+1}}$ can only be the product $s_{l_c,k_c+1}\cdot{}  s_{l_c,k_c+2}\cdot{} \dots{} \cdot{} s_{l_c,k_{c+1}-1}$.
Thus formula~\eqref{keplet} is proved. 

Now, we calculate the length of formula~\eqref{keplet}.  
Formula~\eqref{keplet} is a sum of products.
The length of every product is $n$, thus we need to calculate the number of products. 
Notice, that every product is uniquely determined by the column indices of the terms. 
More exactly we need to know the column indices of the first $n-1$ terms, because the column index of the last term is $j$. 
We can choose these indices from the set $\{i, i+1,\dots, j-1,j\}$.
The order of the selected elements does not matter. 
We only need to determine how many indices are equal to $i$ or to $i+1$, etc, or to $j $.
Thus we need to choose $n-1$ element from a set of $j-i+1$ elements such that repetitions are allowed.
Hence the number of products is 
\begin{align*}
&\binom{n-1 + j-i+1 -1}{n-1} =
\binom{n+j-i-1}{j-i}.
\end{align*}
Since every product has length $n$, the length of $\alpha_{i,j}$ is 
\[ \binom{n+j-i-1}{j-i}\cdot n =  O \left( n^{j-i+1} \right) = O \left( n^m \right). \]
Thus the length of $\alpha_{i,j}$ is at most $O \left( n^{m} \right)$ for every index $1\leq i < j \leq m$.
\end{proof} 

Let  $\szp$ be a semipattern group.
Let $F=T_1T_2\dots{}T_n$ be a polynomial over $\szp$. 
Thus $T_k$ can indicate a constant or a variable over $\szp$ $\left( 1 \leq  k \leq n \right)$.
Of course $T_k$ and $T_l$ can indicate the same constant or variable. 
Let 

\begin{align*}
T_k=
\begin{pmatrix}
 y_{1,k} & x_{1,2,k}& x_{1,3,k} & \dots & x_{1,m,k} \\ 
0 & y_{2,k} & x_{2,3,k} & \dots & x_{2,m,k}  \\ 
0 & 0 & y_{3,k }& \dots & x_{3,m,k} \\  
\vdots & \vdots & \vdots & \ddots & \vdots  \\ 
0 & 0 & 0 & \dots & y_{m,k} 
\end{pmatrix}.
\end{align*} 

If $T_k\in \szp$ indicates constant then $y_{i,k}$ is a constant in $\sg_i \leq \fqm$ and  $x_{i,j,k}$ is a constant in $\fq$ $(1 \leq i < j \leq m)$.
If $T_k\in \szp$ is a variable, then $y_{i,k}$ is a variable, that we can substitute  from $\sg_i$, and $x_{i,j,k}$ is a variable that we can substitute  from $\fq$.
Furthermore $T_k=T_l$  if and only if $y_{i,k}=y_{i,l}$ (for every $1\leq i \leq k$) and $x_{i,j,k}=x_{i,j,l}$ for every 
$1 \leq i < j \leq m$.

We can rewrite the  polynomial $F$ with this notation as
\begin{align*}
F&=
\begin{pmatrix}
 y_{1,1} & x_{1,2,1}& x_{1,3,1} & \dots & x_{1,m,1} \\ 
0 & y_{2,1} & x_{2,3,1} & \dots & x_{2,m,1}  \\ 
0 & 0 & y_{3,1 }& \dots & x_{3,m,1} \\  
\vdots & \vdots & \vdots & \ddots & \vdots  \\ 
0 & 0 & 0 & \dots & y_{m,1} 
\end{pmatrix} 
\begin{pmatrix}
 y_{1,2} & x_{1,2,2}& x_{1,3,2} & \dots & x_{1,m,2} \\ 
0 & y_{2,2} & x_{2,3,2} & \dots & x_{2,m,2}  \\ 
0 & 0 & y_{3,2 }& \dots & x_{3,m,2} \\  
\vdots & \vdots & \vdots & \ddots & \vdots  \\ 
0 & 0 & 0 & \dots & y_{m,2} 
\end{pmatrix}\dots \\
&\dots
\begin{pmatrix}
 y_{1,n} & x_{1,2,n}& x_{1,3,n} & \dots & x_{1,m,n} \\ 
0 & y_{2,n} & x_{2,3,n} & \dots & x_{2,m,n}  \\ 
0 & 0 & y_{3,n }& \dots & x_{3,m,n} \\  
\vdots & \vdots & \vdots & \ddots & \vdots  \\ 
0 & 0 & 0 & \dots & y_{m,n} 
\end{pmatrix}.
\end{align*}

After multiplying these matrices using Lemma~\ref{mxszor} we obtain 

\[
F=
\begin{pmatrix}
 f_{1} & g_{1,2}& g_{1,3} & \dots & g_{1,m} \\ 
0 & f_{2} & g_{2,3} & \dots & g_{2,m}  \\ 
0 & 0 & f_{3}& \dots & g_{3,m} \\  
\vdots & \vdots & \vdots & \ddots & \vdots  \\ 
0 & 0 & 0 & \dots & f_{m}
\end{pmatrix},
\]
where
\begin{align*}
 \notag f_{i}&=\prod_{k=1}^{n}y_{i,k}, \text{ and}\\ 
 g_{i,j}&=
\sum_{b=0}^{j-i-1}    
\sum_{i<l_1<\dots<l_b<j}   
\sum_{k_{b+1}=b+1}^{n}
\sum_{k_{b}=b}^{k_{b+1}-1}    
\dots
\sum_{k_{2}=2}^{k_{3}-1}
\sum_{k_{1}=1}^{k_{2}-1}
\left(\prod_ {c_0=1}^{k_1-1} y_{i,c_0}\right) x_{i,l_1,k_1}
\\
&\left(\prod_ {c_1=k_1+1}^{k_{2}-1} y_{l_1,c_1}\right)
 x_{l_{1},l_{2},k_{2}} 
\dots{} \left( \prod_ {c_{b}=k_{b}+1}^{k_{b+1}-1}     y_{l_{b},c_{b}}\right)
x_{l_b,j,k_{b+1}}       \left( \prod_ {c_{b+1}=k_{b+1}+1}^{n} y_{j,c_{b+1}}\right). \notag
\end{align*}

The polynomial $F$  can attain the unit matrix for  a substitution if and only if $f_i$ attains $1$ ($1 \leq i \leq m$) and $ g_{i,j}$ attains $0$ for the same substitution ($1 \leq i < j \leq m$). 
Thus $F=I$ is solvable over $\szp$ if and only if the  system of equations 
\begin{align*}
f_{i}|_{{\fq}, {\sg}_1, \dots{}, {\sg}_m } &=1 & (1 \leq i \leq j) \\
 g_{i,j}|_{{\fq}, {\sg}_1, \dots{}, {\sg}_m }&=0 & (1 \leq i < j \leq m)
\end{align*} is solvable over $\fq$.
This is a system of equations  over $\fq$ where the polynomials are given as sums of monomials.
Hence we can decide the solvability of this   system of equations in polynomial time by Theorem~\ref{egyrszmo}. 

The rewriting of $F$ over $\szp$ into the system of equations 
\begin{align*}
 f_{i}|_{{\fq}, {\sg}_1, \dots{}, {\sg}_m } &=1  & (1 \leq i \leq j)\\
 g_{i,j}|_{{\fq}, {\sg}_1, \dots{}, {\sg}_m }&=0 & (1 \leq i < j \leq m)
\end{align*} over $\fq$
can be done in $O\left( n^m\right) $ time by Lemma~\ref{mxszor}, 
and the length of each equation is $O \left( n^m \right)$. 
The number of equations is $\frac{m \cdot (m+1)}{2} = O \left( m^2 \right)$, which does not depend on $n$, only on the group $\szp$. 
By Theorem~\ref{egyrszmo} one can decide whether this system has a solution in $O\left( n^{m^3\cdot q}\right) $ time.

\section{The complexity of equation solvability problem over nilpotent matrix rings \label{3.}}

Let $\M$ be a ring of $m \times m$ matrices over $\Zb_{p^{\alpha}}$ where every entry on or below the main diagonal of each matrix in $\M$ is a multiple of $p$. 
Let $F$ be a polynomial over $\M$ given as a sum of monomials.  
Let $T_1 \dots{} T_n$ denote a monomial in the sum.  
Let 
\[
T_k=
\begin{pmatrix}
 a_{1,1,k}\cdot p & s_{1,2,k}& s_{1,3,k} & \dots & s_{1,m,k} \\ 
 a_{2,1,k}\cdot p   &  a_{2,2,k}\cdot p   & s_{2,3,k} & \dots & s_{2,m,k}  \\ 
 a_{3,1,k}\cdot p   &   a_{3,2,k}\cdot p   &   a_{3,3,k} \cdot p  & \dots & s_{3,m,k} \\  
\vdots & \vdots & \vdots & \ddots & \vdots  \\ 
 a_{m,1,k}\cdot p   &   a_{m,2,k}\cdot p   &   a_{m,3,k} \cdot p  & \dots &   a_{m,m,k} \cdot p 
\end{pmatrix}.
\]
If $T_k\in \M$ indicates constant then $a_{i,j,k}, s_{i,j,k}$ are constants in $\Zb_{p^{\alpha}}$.
If $T_k\in \M$ is a variable, then $a_{i,j,k}, s_{i,j,k}$ are variables, that we can substitute  from $\Zb_{p^{\alpha}}$.
Furthermore, $T_k=T_l$  if and only if $a_{i,j,k} \cdot p=a_{i,j,l} \cdot p$ and $s_{i,j,k}=s_{i,j,l}$ for every $i,j \in \{1,\dots ,m\}$.

We can rewrite the  monomial $T_1 \dots T_n$ with this notation as
\begin{align*}
T_1 \dots T_n&=
\begin{pmatrix}
 a_{1,1,1} \cdot p & s_{1,2,1}& s_{1,3,1} & \dots & s_{1,m,1} \\ 
  a_{2,1,1}\cdot p  &   a_{2,2,1}\cdot p   & s_{2,3,1} & \dots & s_{2,m,1}  \\ 
 a_{3,1,1}\cdot p   & a_{3,2,1}\cdot p   &   a_{3,3,1}\cdot p   & \dots & s_{3,m,1} \\  
\vdots & \vdots & \vdots & \ddots & \vdots  \\ 
 a_{m,1,1}  \cdot p &  a_{m,2,1}\cdot p   &   a_{m,3,1} \cdot p  & \dots & a_{m,m,1}  \cdot p 
\end{pmatrix} 
\dots \\
&\dots{}
\begin{pmatrix}
 a_{1,1,n}\cdot p  & s_{1,2,n}& s_{1,3,n} & \dots & s_{1,m,n} \\ 
  a_{2,1,n}\cdot p   &  a_{2,2,n}\cdot p   & s_{2,3,n} & \dots & s_{2,m,n}  \\ 
  a_{3,1,n}\cdot p   &   a_{3,2,n} \cdot p  & a_{3,3,n}  \cdot p & \dots & s_{3,m,n} \\  
\vdots & \vdots & \vdots & \ddots & \vdots  \\ 
a_{m,1,n} \cdot p  &  a_{m,2,n} \cdot p  &   a_{m,3,n}  \cdot p & \dots &   a_{m,m,n} \cdot p 
\end{pmatrix}.
\end{align*}

After multiplying these matrices we obtain 

\[
T_1 \dots{} T_n=
\begin{pmatrix}
 g_{1,1} & g_{1,2} & g_{1,3} & \dots & g_{1,m} \\ 
 g_{2,1} &  g_{2,2} & g_{2,3} & \dots & g_{2,m}  \\ 
 g_{3,1} &  g_{3,2} &  g_{3,3}& \dots & g_{3,m} \\  
\vdots & \vdots & \vdots & \ddots & \vdots  \\ 
 g_{m,1} &  g_{m,2} &  g_{m,3} & \dots &  g_{m,m}
\end{pmatrix}.
\]

We could compute the expressions for every $g_{i,j}$, similarly as in Lem\-ma~\ref{mxszor}. 
However, we do not need to know the actual formulas, we only need to understand how they look like. 

First $g_{i,j}$ is a polynomial given as a sum of monomials over $\mathbb{Z}_{p^{\alpha}}$.   
In every such monomial we multiply one term from each matrix $T_k$ $(1 \leq k \leq n)$.
Hence the length of every monomial is $n$. 
There are at most $ m^{n-1}$ monomials in every polynomial $g_{i,j}$ according to the usual multiplication of matrices. 
However, every nonzero monomial contains at most $\alpha-1$ terms from on or below the main diagonal, since the characteristic of $\mathbb{Z}_{p^{\alpha}}$ is $p^{\alpha}$.
Therefore  every nonzero monomial is of the form 
\begin{align*}
&\underbrace{s_{i,i_1,1} \cdot s_{i_i,i_2,2} \dots{} s_{i_{k_1-1},i_{k_1},k_1}}_{k_1 \text{- many terms }}\cdot a_{i_{k_1},l_{1},{k_1+1}} \cdot p \cdot
\underbrace{{s}_{l_{1},\bar{i}_{1},{k_1+2}} \cdot s_{\bar{i}_1,\bar{i}_2,{.}} \dots{} {s}_{\bar{i}_{k_2-1},\bar{i}_{k_2},.}}_{k_2 \text{- many terms }} \\
&\cdot a_{\bar{i}_{k_2},l_{2},{.}}  \cdot p \dots  
a_{\hat{i}_{k_{b-1}},l_{b-1},.} \cdot p  \cdot 
\underbrace{s_{l_{b-1},\tilde{i}_{1},.} \dots{} {s}_{\tilde{i}_{k_{b}-1},\tilde{i}_{k_{b}},n}}_{k_b \text{- many terms }}
\end{align*}
where $1 \leq b \leq \alpha $ and $0 \leq k_1, k_2, \dots k_b$. 
Notice that $k_1, k_2, \dots k_b \leq m-1$ holds as well. 
Indeed, every term $s_{i,j,k}$ is from above the main diagonal, hence $i <  j$.
Therefore $1 < i_1 < i_2 < \dots < i_{k_1} \leq m $ and thus $k_1\leq m-1$. 
Similarly $k_2, \dots, k_b \leq m-1$. 
Hence the length of every nonzero monomial is at most $(m-1) \cdot \alpha + \alpha - 1= m \alpha -1$.
In particular, if $m \alpha -1 < n$ then every monomial in $g_{i,j}$ equals $0$. 
Therefore there are at most $m ^{m \alpha -2}$ monomials in every polynomial $g_{i,j}$ and thus $||g_{i,j}||\leq ( m \alpha -1) \cdot m ^{m \alpha -2}\leq \alpha\cdot m ^{m  \alpha-1 }$.

The input polynomial $F$ is a sum of at most $||F||$ monomials $T_1\dots T_n$. 
Let 
\[
F=
\begin{pmatrix}
 f_{1,1} & f_{1,2} & f_{1,3} & \dots & f_{1,m} \\ 
 f_{2,1} &  f_{2,2} & f_{2,3} & \dots & f_{2,m}  \\ 
 f_{3,1} &  f_{3,2} &  f_{3,3}& \dots & f_{3,m} \\  
\vdots & \vdots & \vdots & \ddots & \vdots  \\ 
 f_{m,1} &  f_{m,2} &  f_{m,3} & \dots & f_{m,m}
\end{pmatrix}.
\] 
Then $f_{i,j}$ is the sum of at most $||F||$-many polynomials $g_{i,j}$ for every $i,j \in \{1,\dots m\}$.
Thus $||f_{i,j}|| \leq ||F|| \cdot \alpha m ^{m  \alpha-1 } = O \left( \left|\left| F \right|\right| \right)$, 
as $\alpha$ and $m$ depend only on the ring $\R$. 

The polynomial $F$  can attain the zero matrix for  a substitution if and only if $f_{i,j}$ attains zero for the same substitution for every $i,j \in \{ 1, \dots{}, m\}$. 
Thus $ F=0$ is solvable over $\M$ if and only if the  system of equations 
\begin{align*}
 & &  f_{i,j}=0  &  &i,j \in \{ 1, \dots{}, m\} & &
\end{align*}
is solvable over $\mathbb{Z}_{p^{\alpha}}$.
This is a system of equations over $\za$ where the polynomials are given as sums of monomials.
Hence we can decide the solvability of this system of equations in polynomial time by Theorem~\ref{er}.
 
The rewriting of $F$ over $\M$ into the system of equations 
\begin{align*}
& &  f_{i,j}=0  &  &i,j \in \{ 1, \dots{}, m\} & &
\end{align*}
can be done in $O\left(  ||F||\right) $ time, 
and $\left|\left| f_{i, j} \right| \right| = O \left( \left| \left| F \right| \right| \right)$. 
The number of equations is $m^2$, which does not depend on $||F||$, only on the ring $\M$. 
By Theorem~\ref{er} one can decide whether this system has a solution in $O\left( ||F||^{\al^2  \cdot m^2 \cdot p^{ 2\alpha^2 }}\right) $ time. 

Finally, 
we explain how this result can be applied to determine if an equation $f=0$ is solvable over an \emph{arbitrary} nilpotent ring $\R$ of characteristic $p^\alpha$. 
Let $\M$ be the matrix ring as in Theorem~\ref{wil}, 
and let $\R \cong \M / \I$.
Further, let $F$ denote the polynomial in $\M$ corresponding to $f$. 
Notice, that $p^\al$ was the characteristic of $\R$ in Theorem~\ref{wil}, hence $\al = O \left( \log|\R| \right) $.
Furthermore $m =  O \left( \log|\R| \right) $ and $p^{\al}= O \left( |\R| \right)$.
Therefore $\al^2  \cdot m^2 \cdot p^{ 2\alpha^2 }=  O \left(  |\R|^{2\log \R} \log^4|\R| \right)$.
Hence, one can decide if $F=0$ is solvable over $\R \cong \M/\I$ in $O\left( ||F||^{|\R|^{2\log \R} \log^4|\R|}\right) $ time. 
Thus, one can decide if the equation $f=0$ is solvable over 
$\R$ in $O\left( \left|\I\right| \cdot ||F||^{|\R|^{2\log \R} \log^4|\R|}\right) $ time. 
Now, $\left| \I \right| \leq \left| \M \right| \leq \left(p^\alpha\right)^{m^2} = O \left( \left| \R \right|^{\log^2 \left| \R \right|} \right)$, 
which only depends on $\R$. 
Therefore, 
if $f$ is given as a \emph{sum of monomials}, 
then $\left| \left| F \right| \right| = O \left( \left| \left| f \right| \right| \right)$, 
and 
one can decide $ f=0$ over an arbitrary nilpotent ring 
$\R$ in $O\left( ||f||^{|\R|^{2\log \R} \log^4|\R| }\right) $ time, 
as well. 
If, however, 
$f$ is an arbitrary polynomial over $\R$, 
then after rewriting it as a sum of monomials we have 
$\left| \left| F \right| \right| = O \left( \left| \left| f \right| \right|^{\log \left| \R \right|} \right)$, 
giving an extra $\log \left| \R \right|$ factor in the exponent. 
Thus, 
one can decide $ f=0$ over an arbitrary nilpotent ring 
$\R$ in $O\left( ||f||^{|\R|^{2\log \R} \log^5|\R|}\right) $ time.

\newpage


\end{document}